\documentclass[11pt]{amsart}

\usepackage[bookmarks=true]{hyperref}

\usepackage{mathptmx}
\usepackage{amsmath}
\usepackage{amssymb}
\usepackage{amsthm}
\usepackage{array}
\usepackage[all,tips]{xy}
\usepackage{enumerate}
\usepackage{graphicx}
\usepackage{lmodern}
\usepackage{mathrsfs}
\usepackage{plain}

\usepackage{float}

\usepackage[toc,page]{appendix} 
\usepackage{chngcntr}
\newtheorem{thm}{Theorem}
\newtheorem*{thm*}{Theorem}
\newtheorem{lem}{Lemma}

\theoremstyle{definition}

\newtheorem{rmk}{Remark}
\newtheorem*{rmk*}{Remark}
\newtheorem{example}{Example}
\newtheorem*{notice*}{Notice}
\newtheorem*{epitome*}{Epitome of growth comparison}

\newtheorem{theorem}{Theorem}
\renewcommand*{\thetheorem}{\Alph{theorem}}

\DeclareMathOperator{\SO}{SO}

\DeclareMathOperator{\Gr}{Gr}

\renewcommand{\qedsymbol}{$\square$}

\def\loba{\loba}

\def\<{\langle}
\def\>{\rangle}

\def\gd{\hbox{$\bullet\kern-4pt$ --\kern-4pt--\kern-4pt-- 
\kern-4pt}}
\def\gr #1{\hbox{$\bullet\kern-4pt $ --\kern-3pt\raise
6pt\hbox{$\scriptstyle #1$}\kern-8pt --\kern-6pt ---\kern-5pt --\kern-1pt}}
\def\grastrich #1{\hbox{$\bullet\kern-2pt $ ---\kern-3pt\raise
8pt\hbox{$\alpha'_#1$}\kern-7pt --\kern-6pt ---\kern-5pt ---\kern+1pt}}
\def\grinf{\hbox{$\bullet\kern-4pt $ --\kern-4pt\raise
6pt\hbox{$\infty$}\kern-11pt --\kern-6pt ---\kern-5pt --}}
\def\grh #1{\hbox{$\bullet\kern-2pt $ ---\kern-3pt\raise
8pt\hbox{$#1$}\kern-7pt --\kern-6pt ---\kern-5pt ---\kern+1pt}}
\def\gra #1{\hbox{$\bullet\kern -2pt$ --\raise
8pt\hbox{$\alpha_{#1}$}\kern-13pt
---\kern-4pt\kern-4pt ---\kern-5pt ---\kern +2pt}}
\def\graa #1{\hbox{$\bullet\kern -2pt$ ---\kern-6pt\raise
8pt\hbox{$\alpha_{#1}$}\kern-14pt
---\kern-4pt\kern-4pt ---\kern-5pt ---\kern +2pt}}
\def\graablack #1{\hbox{$\bullet\kern -2pt$ ---\kern-6pt\raise
8pt\hbox{$\alpha_{#1}$}\kern-14pt
---\kern-4pt\kern-4pt ---\kern-5pt ---\kern +2pt$\bullet$}}
\def\diamondgra #1{\hbox{$\diamond\kern -2pt$ --\raise
8pt\hbox{$\alpha_{#1}$}\kern-15pt
---\kern-4pt\kern-4pt ---\kern-5pt ---\kern +2pt}}
\def\gradiamond #1{\hbox{$\circ\kern -2pt$ --\raise
8pt\hbox{$\alpha_{#1}$}\kern-13pt
---\kern-4pt\kern-4pt ---\kern-5pt ---$\diamond$}}
\def\diamondgr #1{\hbox{$\diamond\kern -2pt$ ---\raise
8pt\hbox{${#1}$}\kern-13pt
---\kern-4pt\kern-4pt ---\kern-5pt ---\kern +2pt}}
\def\grdiamond #1{\hbox{$\circ\kern -2pt$ --\raise
8pt\hbox{${#1}$}\kern-13pt
---\kern-4pt\kern-4pt ---\kern-5pt ---\kern+2pt$\diamond$}}
\def\blackgra #1{\hbox{$\bullet\kern -2pt$ --\raise
8pt\hbox{$\alpha_{#1}$}\kern-13pt
---\kern-4pt\kern-4pt ---\kern-5pt ---\kern +2pt}}
\def\grablack #1{\hbox{$\circ\kern -2pt$ --\raise
8pt\hbox{$\alpha_{#1}$}\kern-13pt
---\kern-4pt\kern-4pt ---\kern-5pt ---\kern +1pt$\bullet$}}
\def\blackgr #1{\hbox{$\bullet\kern-2pt $ ---\kern-3pt\raise
8pt\hbox{$#1$}\kern-7pt --\kern-6pt ---\kern-5pt ---\kern+1pt}}
\def\grblack #1{\hbox{$\circ\kern-2pt $ ---\kern-3pt\raise
8pt\hbox{$#1$}\kern-7pt --\kern-6pt ---\kern-5pt ---\kern+1pt$\bullet$}}
\def\grslash{\hbox{$\bullet$\kern-1pt\raise8pt\hbox{$\diagup$}\raise14pt\hbox{$\bullet$}\raise8pt\hbox{$\diagdown$} \kern-3pt$\bullet$\kern0pt---\kern-3pt---\kern0pt$\bullet$}}
\def\gra3tilde{\hbox{$\bullet$\kern-2pt\raise8pt\hbox{$\diagup$}\kern-1pt\raise14pt\hbox{$\bullet$}\raise8pt\hbox{$\diagdown$} \kern-5pt$\bullet$\kern-28pt---\kern-3pt---\kern-2pt-}}

\def\loba{\hbox{\rm J\kern-1pt I}}

\def\r#1{\hbox{$#1$}\raise-8pt\hbox{${\,\tilde{}}$}}

\begin{document}

\title[\resizebox{4.6in}{!}
{Hyperbolic Coxeter groups and minimal growth rates in dimensions four and five}]{Hyperbolic Coxeter groups and minimal 
growth rates in dimensions four and five}

\author{Naomi Bredon}
\address{Department of Mathematics\\
University of Fribourg\\
CH-1700 Fribourg\\Switzerland}
\email{naomi.bredon@unifr.ch}

\author{Ruth Kellerhals}
\address{Department of Mathematics\\
University of Fribourg\\
CH-1700 Fribourg\\Switzerland}
\email{ruth.kellerhals@unifr.ch}

\thanks{Naomi Bredon and Ruth Kellerhals are partially supported by the Swiss National Science Foundation 200021--172583}




\begin{abstract}    

For small $n$, the known compact hyperbolic $n$-orbifolds 
of minimal volume are intimately related to Coxeter groups of smallest rank. For $n=2$ and $3$, these Coxeter groups are given by the triangle group $[7,3]$ and the tetrahedral group $[3,5,3]$, and they are also distinguished by the 
fact that they have minimal growth rate among {\it all} cocompact hyperbolic Coxeter groups in $\hbox{Isom}\mathbb H^n$, respectively. In this work, we consider the cocompact Coxeter simplex group
$G_4$ with Coxeter symbol $[5,3,3,3]$ in $\hbox{Isom}\mathbb H^4$ and 
the cocompact Coxeter prism group $G_5$ based on $[5,3,3,3,3]$
in $\hbox{Isom}\mathbb H^5$. Both groups are arithmetic and related to the fundamental group of the minimal volume arithmetic compact hyperbolic $n$-orbifold for $n=4$ and $5$, respectively.
Here, we prove that the group $G_n$ is distinguished by having smallest growth rate among all Coxeter groups
acting cocompactly on $\mathbb H^n$ for $n=4$ and $5$, respectively. The proof is based on combinatorial properties of compact hyperbolic Coxeter polyhedra, some partial classification results and certain
monotonicity properties of growth rates of the associated Coxeter groups.

\end{abstract}


\maketitle
{\it In memoriam Ernest B.  Vinberg}
\vskip1cm
\section{Introduction}\label{Intro}
Let $\mathbb H^n$ denote the real hyperbolic $n$-space and $\hbox{Isom}\mathbb H^n$ its isometry group. A hyperbolic Coxeter group $G\subset\hbox{Isom}\mathbb H^n$ of rank $N$ is a cofinite discrete group generated by $N$ reflections with respect to hyperplanes in $\mathbb H^n$. Such a  group corresponds to a finite volume Coxeter polyhedron $P\subset\mathbb H^n$ with $N$ facets, which in turn is a convex polyhedron all of whose dihedral angles are of the form $\pi/k$ for an integer $k\ge2$. Hyperbolic Coxeter groups are geometric realisations of abstract Coxeter systems $(W,S)$
consisting of a group $W$ with a finite set $S$ of generators satisfying the relations $s^2=1$ and $(ss')^{m_{ss'}}=1$ where $\,m_{ss'}=m_{s's}\in\{2,3,\ldots,\infty\}$ for $s\not=s'$. For small rank $N$, the group $W$ is characterised most conveniently by its Coxeter symbol or its Coxeter graph.

Hyperbolic Coxeter groups are not only characterised by a simple presentation but they are also distinguished in other ways. For example, for small $n$, they appear as fundamental groups of smallest volume orbifolds $O^n=\mathbb H^n/\Gamma$ where $\Gamma\subset\hbox{Isom}\mathbb H^n$ is a discrete subgroup (see 
\cite{Belo1,Belo2}, \cite{Hild}, \cite{Martin}, \cite{EK} and \cite{K}, for example).
In particular, for $n=2$ and $3$, the compact hyperbolic $n$-orbifold of minimal volume is the quotient of $\mathbb H^n$ 
by a Coxeter group of smallest rank and given by the triangle group
$[7,3]$ and
the $\mathbb Z_2$-extension of the tetrahedral group $[3,5,3]$. For $n=4$ and $5$, and by restricting to the arithmetic context, the compact hyperbolic $n$-orbifold of minimal volume is the
quotient of $\mathbb H^n$ by the $4$-simplex group $[5,3,3,3]$
and by the Coxeter $5$-prism group based on $[5,3,3,3,3]$, respectively.

In parallel to volume we are interested in the spectrum of small growth rates of hyperbolic Coxeter groups $G=(W,S)$. In general, the growth series $f_S(t)$ of a Coxeter system $(W,S)$ is 
given by 
\[
f_S(t)=1+\sum\limits_{k\ge1}a_kt^k\,\,,
\]
where $a_k\in\mathbb Z$ is the number of elements $w\in W$ with $S$-length $k$. The series $f_S(t)$ 
can be computed by Steinberg's formula
\begin{equation*}
\frac{1}{f_S(t^ {-1})}=\sum\limits_{{W_T<W\atop\scriptscriptstyle{{\vert W_T\vert<\infty}}}}\,
\frac{(-1)^{\vert T\vert}}{f_T(t)}\,,
\end{equation*}
where $W_T\,,\,T\subset S\,,$ is a finite Coxeter subgroup of $W$, and where $W_{\varnothing}=\{1\}$. In particular, $f_S(t)$ is a rational function 
that can be expressed as the quotient of coprime monic polynomials $p(t), q(t)\in\mathbb Z[t]$ of equal degree. 
For cocompact hyperbolic Coxeter groups, the series $f_S(t)$ is infinite and has radius of convergence $R<1$ which can be identified with the real algebraic integer given by the smallest positive root of the denominator polynomial $q(t)$. The growth rate 
$\tau_G=\tau_{(W,S)}$ is defined by
\begin{equation*}
\tau_G=\limsup\limits_{k\rightarrow\infty}\sqrt[\leftroot{0}\uproot{2}k]{a_k}\,\,,
\end{equation*}
and $\tau_G$ coincides with the inverse of the radius of convergence $R$ of $f_S(t)$. In contrast to the finite and affine cases, hyperbolic Coxeter groups are of exponential growth.

In \cite{Hironaka} and \cite{KK}, it is shown that the triangle group $[7,3]$ and the tetrahedral group $[3,5,3]$ have minimal growth rate among {\it all} cocompact hyperbolic Coxeter groups in $\hbox{Isom}\mathbb H^n$ for $n=2$ and $3$, respectively.
These results have an interesting number theoretical component since the growth rate $\tau$ of any Coxeter group acting cocompactly on $\mathbb H^n\,$ for $n=2$ and $3$ is either a quadratic unit or a Salem number, that is, $\tau$ is a real algebraic integer $\alpha>1$ whose inverse
is a conjugate of $\alpha$, and all other conjugates lie on the unit circle. 
In particular, the growth rate $\tau_{[7,3]}$ equals
the smallest known Salem number, and it is given by Lehmer's number
$\alpha_L\approx1.17628$ with minimal polynomial 
$L(t)=t^{10}+t^9-t^7-t^6-t^5-t^4-t^3+t+1$. The constant $\alpha_L$ plays an important role in the strong version of Lehmer's problem  about a universal lower bound for Mahler measures of nonzero noncyclotomic irreducible integer polynomials; see \cite{Smyth}.

The proof in \cite{KK} of the two results above is based on the fact that for $n=2$ and $3$ the rational function $f_S(t)$ comes with an explicit formula in terms of the exponents of the Coxeter group $G=(W,S)\subset\hbox{Isom}\mathbb H^n$. 

For dimensions $n\ge4$, however, there are only a few structural results, and closed formulas for growth functions do not exist in general.
In this work, we establish the following results for $n=4$ and $5$ by developing a new proof strategy.
\begin{theorem}\label{thm:A}
Among all Coxeter groups acting cocompactly on $\mathbb H^4$, the 
Coxeter simplex group $[5,3,3,3]$ has minimal growth rate, and as such it is unique.
\end{theorem}

The cocompact Coxeter prism group based on $[5,3,3,3,3]$ in $\hbox{Isom}\mathbb H^5$ was first discovered by Makarov \cite{Makarov} and arises as the discrete group generated by the reflections in the compact straight Coxeter prism $M$ with base $[5,3,3,3]$. More concretely, the prism $M$ is the truncation of the (infinite volume) Coxeter $5$-simplex $[5,3,3,3,3]$ by means of the polar hyperplane associated to its ultra-ideal vertex characterised by the vertex simplex $[5,3,3,3]$. Our second result can be stated as follows.

\begin{theorem}\label{thm:B}
Among all Coxeter groups acting cocompactly on $\mathbb H^5$, the 
Coxeter prism group based on $[5,3,3,3,3]$ has minimal growth rate, and as such it is unique.
\end{theorem}


\vspace{2mm}
The work is organised as follows. In Section \ref{section2-1} we provide the necessary background about hyperbolic Coxeter polyhedra, their reflections groups and the characterisation by means of the Vinberg graph and the Gram matrix. We present the relevant classification results for families of Coxeter polyhedra with few facets due to Esselmann, Kaplinskaja and Tumarkin. Of particular importance is the structural result, presented in Theorem \ref{1-dotted} and due Felikson and Tumarkin, about the existence of non-intersecting facets of a compact Coxeter polyhedron.
In Section \ref{section2-2}, we consider abstract Coxeter systems with their Coxeter graphs and Coxeter symbols and introduce the notions of growth series and growth rates.  Another important ingredient is the growth monotonicity result of Terragni as given in Theorem \ref{Terragni}.
The proofs of our results are presented in Section \ref{section3}.
The proof of Theorem \ref{thm:A} is based on a simple growth rate comparison argument and serves as an inspiration how to attack the proof of Theorem \ref{thm:B}. To this end, we establish Lemma \ref{compare1} and Lemma \ref{compare2} about the comparison of growth rates of certain Coxeter groups of rank $4$.
Then, we consider compact Coxeter polyhedra in $\mathbb H^5$ in terms of the number $N\ge6$ of their facets.
Since compact hyperbolic Coxeter $n$-simplices exist only for $n\le4$, we look at compact Coxeter polyhedra $P\subset\mathbb H^5$
with $N=7$, $N=8$ and $N\ge9$ facets, respectively. Certain classification results help us dealing with the cases $N=7$ and $8$ while for $N\ge9$, we look for particular subgraphs in the Coxeter graph of $P$ and conclude by means of Lemma \ref{compare1}, Lemma \ref{compare2} and Theorem \ref{Terragni}.

\vspace{10pt}
\noindent
{\em Acknowledgement.\quad}The authors would like to thank Yohei Komori for helpful comments on an earlier version of the paper.
\section{Hyperbolic Coxeter polyhedra and growth rates}\label{section2}
\subsection{Hyperbolic Coxeter polyhedra and their reflection groups}\label{section2-1}
Denote by $\mathbb H^n$ the standard hyperbolic $n$-space realised by the upper sheet of the hyperboloid in $\mathbb R^{n+1}$ according to 
\[
\mathbb H^n=\{x\in\mathbb R^{n+1}\mid q_{n,1}(x)=x_1^2+\dots+x_n^2-x_{n+1}^2=-1\,,\,x_{n+1}>0\}\,.
\]
A hyperbolic hyperplane $H$ is the intersection of a vector subspace of dimension $n$ with $\mathbb H^n$ and 
can be represented as the Lorentz-orthogonal complement $H=e^L$ by means of a vector $e$ of (normalised) Lorentzian norm $q_{n,1}(e)=1$.
The isometry group $\hbox{Isom}\mathbb H^n$ of $\mathbb H^n$ is given by the group $\mathrm{O}^+(n,1)$ of positive Lorentzian matrices
leaving the bilinear form $\langle x,y\rangle_{n,1}$ associated to $q_{n,1}$ and the upper sheet invariant. It is well known that $\mathrm{O}^+(n,1)$
is generated by linear reflections $r=r_H\,:\,x\mapsto x-2\,\langle e,x\rangle_{n,1}\,e\,$ with respect to hyperplanes 
$H=e^L$ (see \cite[Section A.2]{Ben-Petr}).

A hyperbolic $n$-polyhedron $P\subset \mathbb H^n$ is the non-empty intersection of a finite number $N\ge n+1$ of half-spaces $H_i^-$ bounded by hyperplanes
$H_i$ all of whose normal unit vectors $e_i$ are directed outwards with respect to $P$, say. A facet of $P$ is the intersection of $P$ with
one of the hyperplanes $H_i\,,\,1\le i\le N$.
A polyhedron is a \emph{Coxeter polyhedron} if all of its dihedral angles are of the form $\frac{\pi}{k}$ for an integer $k\ge2$. 

In this work, we suppose that $P$ is a \emph{compact} hyperbolic Coxeter polyhedron so that
$P$ is the convex hull of finitely many points in $\mathbb H^n$. 
In particular, $P$ is \emph{simple} since all dihedral angles of $P$ are less than or equal to $\pi/2$. 
As a consequence,
each vertex $p$ of $P$ is the intersection of $n$ hyperplanes bounding $P$ and characterised by a vertex neighborhood 
which is a cone over a spherical Coxeter $(n-1)$-simplex. 

The following structural result of A. Felikson and P. Tumarkin \cite[Theorem A]{FT1} 
will be of importance later in this work.   
For its statement, the compact Coxeter polyhedra in $\mathbb H^4$ that are products of two simplices of dimensions greater than $1$ will play a certain role. There are seven such polyhedra which were discovered by F.  Esselmann \cite{Essel} (see also \cite{F-web} and Examples \ref{n+2}, \ref{Esselmann} and \ref{ex:tau-Esselmann}).

\begin{thm}\label{1-dotted}
Let $P\subset\mathbb H^n$ be a compact Coxeter polyhedron. If $n\le 4$ and all facets of $P$ are mutually intersecting, then $P$ is either a simplex or one of the seven Esselmann polyhedra. If $n>4$, then $P$ has a pair of non-intersecting facets.
\end{thm}

Fix a compact Coxeter polyhedron $P\subset\mathbb H^n$ with its bounding hyperplanes $H_1,\ldots,H_N$ as above.
Denote by $G$ the group generated by the reflections $r_i=r_{H_i}\,,\,1\le i\le N$. Then, $G$ is a cocompact discrete subgroup of $\hbox{Isom}\mathbb H^n$
with $P$ equal to the closure of a fundamental domain for $G$. The group $G$ is called a \emph{(cocompact) hyperbolic Coxeter group}.
It follows that $G$ is finitely presented with natural generating set $S=\{r_1,\ldots,r_N\}$ and relations
\begin{equation}\label{refl-relations}
 r_i^2=1\quad\hbox{and}\quad(r_ir_j)^{m_{ij}}=1\,\,,
\end{equation}
where $m_{ij}=m_{ji}\in\{2,3,\ldots,\infty\}$ for $\,i\not= j$. Here, $m_{ij}=\infty$ means that the product $r_ir_j$ is of 
infinite order which fits into the following geometric picture. Denote by 
$\Gr(P)=\big(\langle e_i,e_j\rangle_{n,1}\big)\in\hbox{Mat}(N;\mathbb R)$ the Gram matrix of $P$.  Then, the coefficients 
of $\Gr(P)$ off its diagonal can be interpreted as follows.
\begin{equation}
\label{eq:Gram-coeff}
-\langle e_i,e_j\rangle_{n,1}=\left\{
\begin{array}{ll}
\cos\frac{\pi}{m_{ij}}&\textrm{if $\measuredangle(H_i,H_j)=\frac{\pi}{m_{ij}}\,$};\\
\cosh l_{ij}&\textrm{if $d_{\mathbb H}(H_i,H_j)=l_{ij}>0\,$}.\\
\end{array}
\right.
\end{equation}
The matrix $\Gr(P)$ is of signature $(n,1)$. Furthermore, it contains important information about $P$. For example, each vertex of $P$ is 
characterised by a positive definite $n\times n$ principal submatrix of $\Gr(P)$.

Beside the Gram matrix $\Gr(P)$, the Vinberg graph $\Sigma(P)$ is very useful to describe a Coxeter polyhedron $P$ (and its associated reflection group $G$)
if the number $N$ of its facets is small
in comparison with the dimension $n$.  The Vinberg graph $\Sigma(P)$ consists of nodes $v_i\,,\,1\le i\le N,$ which correspond to the hyperplanes 
$H_i$ or their unit normal vectors $e_i$. The number $N$ of nodes is called the \emph{order} of $\Sigma(P)$.
If the hyperplanes $H_i$ and $H_j$ are not orthogonal, the corresponding 
nodes $v_i$ and $v_j$ are connected by an edge with weight $m_{ij}\ge3$ if $\measuredangle(H_i,H_j)=\frac{\pi}{m_{ij}}\,$; they are connected 
by a dotted edge (sometimes with weight $l_{ij}$) if $H_i$ and $H_j$ are at distance $l_{ij}>0$ in $\mathbb H^n$.
The weight $m_{ij}=3$ is omitted since it occurs very frequently.

Since $P$ is compact (and hence of finite volume), the Vinberg graph $\Sigma(P)$ is connected. Furthermore, by deleting a node together with the edges emanating from it so that $\Sigma(P)$ gives rise to two connected components $\Sigma_1$ and $\Sigma_2$, at most one of two subgraphs $\Sigma_1,\Sigma_2$ can have a dotted edge (since otherwise, the signature condition of $\Gr(P)$ is violated).



The subsequent examples summarise the classification results for compact Coxeter $n$-polyhedra in terms of the number $N=n+k\,,\,1\le k\le 3\,,$ of their facets.
\begin{example}\label{ex1}
The compact hyperbolic Coxeter simplices were classified by Lann\'er \cite{Lanner} and exist for $n\le 4$, only. In the case $n=4$, there are precisely five simplices $L_i$ whose 
Vinberg graphs $\Sigma_i=\Sigma(L_i)\,,\,1\le i\le 5\,,$ are given in Figure \ref{fig:cox-simplex}. The simplex $L=L_1$ described by the top left 
Vinberg graph (or by its Coxeter symbol $[5,3,3,3]$; see Section \ref{section2-2} and \cite{JKRT}) will
be of particular importance.
\end{example}
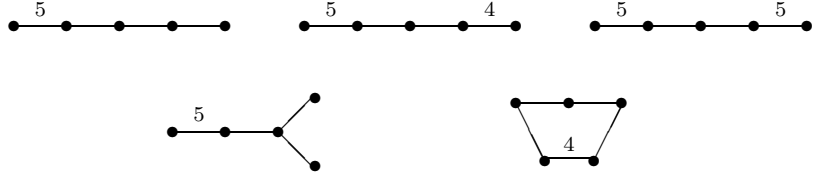
\begin{figure}
\centering
\begin{picture}(300,60)
    \put(0,45){\circle*{3.5}}
    \put(20,45){\circle*{3.5}}
    \put(40,45){\circle*{3.5}}   
    \put(60,45){\circle*{3.5}} 
    \put(80,45){\circle*{3.5}}

    \put(110,45){\circle*{3.5}}
    \put(130,45){\circle*{3.5}}
    \put(150,45){\circle*{3.5}}   
    \put(170,45){\circle*{3.5}} 
    \put(190,45){\circle*{3.5}}
       
    \put(220,45){\circle*{3.5}}
    \put(240,45){\circle*{3.5}}
    \put(260,45){\circle*{3.5}}   
    \put(280,45){\circle*{3.5}} 
    \put(300,45){\circle*{3.5}}
    
    \put(0,45){\line(1,0){80}}
    \put(110,45){\line(1,0){80}}
    \put(220,45){\line(1,0){80}}

    \put(8,49){$\scriptstyle 5$}
    \put(118,49){$\scriptstyle 5$} 
    \put(178,49){$\scriptstyle 4$}  
    \put(228,49){$\scriptstyle 5$}   
    \put(288,49){$\scriptstyle 5$}
    
    \put(60,5){\circle*{3.5}}
    \put(80,5){\circle*{3.5}}
    \put(100,5){\circle*{3.5}}   
    \put(114,18){\circle*{3.5}} 
    \put(114,-8){\circle*{3.5}} 
    
    \put(60,5){\line(1,0){40}}
    \put(100,5){\line(1,1){14}}
    \put(100,5){\line(1,-1){14}}
    
    \put(68,9){$\scriptstyle 5$}
    \put(208,-2){$\scriptstyle 4$} 
    
    \put(190,16){\circle*{3.5}}     
    \put(210,16){\circle*{3.5}}
    \put(230,16){\circle*{3.5}}
    \put(201,-6){\circle*{3.5}}   
    \put(219.5,-6){\circle*{3.5}} 
    
    \put(190,16){\line(1,0){40}}
    \put(200,-5){\line(1,0){20}}
    \put(190,16){\line(1,-2){11}} 
    \put(220,-5){\line(1,2){11}}  
    
\end{picture}
\bigskip
\bigskip

\caption{The compact Coxeter simplices in $\mathbb H^4$}
\label{fig:cox-simplex}
\end{figure}
\begin{example}\label{n+2}
The compact Coxeter polyhedra with $n+2$ facets in $\mathbb H^n$ have been classified. The list consists of the 7 examples of Esselmann
and the (glueings of) straight Coxeter prisms due to I. Kaplinskaja (see \cite{F-web}, \cite{Perren}, for example). The examples of Esselmann are products of two simplices of dimensions bigger than 1 and exist in $\mathbb H^4$, only.  The prisms (and their glueings) of Kaplinskaja exist for $n\le 5$, and the list includes the Makarov prism $M$ based on $[5,3,3,3,3]$ (see Theorem \ref{thm:B}). Observe that the Vinberg graphs of all Kaplinskaja examples (including their glueings) contain one dotted edge.
\begin{figure}[H]
\centering
\begin{picture}(208,40)
    \put(0,0){\circle*{3.5}}
    \put(0,30){\circle*{3.5}}
    \put(16,15){\circle*{3.5}}   
    \put(41,15){\circle*{3.5}} 
    \put(56,0){\circle*{3.5}}    
    \put(56,30){\circle*{3.5}} 
            
    \put(17,15){\line(1,0){23}}
    \put(0,0){\line(0,1){30}}
    \put(56,0){\line(0,1){30}}  
    \put(0,0){\line(1,1){16}}
    \put(0,30){\line(1,-1){16}} 
    \put(40,15){\line(1,1){16}}
    \put(40,15){\line(1,-1){16}}   
    
    \put(27,18){$\scriptstyle 4$}
    \put(-8,12){$\scriptstyle 4$}
    \put(58,12){$\scriptstyle 4$}

    \put(120,0){\circle*{3.5}}
    \put(120,30){\circle*{3.5}} 
    \put(152,0){\circle*{3.5}} 
    \put(152,30){\circle*{3.5}} 
    \put(168,15){\circle*{3.5}} 
    \put(188,15){\circle*{3.5}}      
    \put(208,15){\circle*{3.5}}

     \put(120,0){\line(1,0){32}}
     \put(120,30){\line(1,0){32}}
     \put(120,0){\line(0,1){30}}      
     \put(152,30){\line(1,-1){16}} 
     \put(152,0){\line(1,1){16}}
     \put(168,15){\line(1,0){20}}
     \put(192,12.5){$\cdots$}
     \put(112,12){$\scriptstyle 4$}
     
     \put(25,-10){$E$}
     \put(173,-10){$K$}

\end{picture}
\vskip.8cm

\caption{The Vinberg graphs of an Esselmann polyhedron $E\subset \mathbb H^4$
 and of a Kaplinskaja prism $K\subset\mathbb H^5$
}
\label{fig:n+2}
\end{figure}
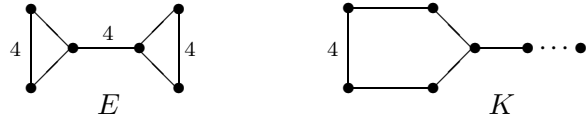
\end{example}

\begin{example}\label{n+3}
The compact hyperbolic Coxeter polyhedra $P\subset\mathbb H^n\,,\,n\ge4\,,$ with $n+3$ facets exist up to $n=8$ and have been enumerated by Tumarkin \cite{Tum}. 
For $n=5$, his list comprises 16 polyhedra, and they are characterised by Vinberg graphs with exactly 3 (consecutive) dotted edges, up to the 
exceptional case $T\subset \mathbb H^5$. The polyhedron $T$ has exactly one pair of non-intersecting facets and is depicted in Figure \ref{fig:Tumarkin}.
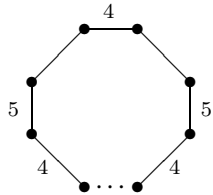
\begin{figure}[H]
\centering
\begin{picture}(60,70)
    \put(0,20){\circle*{3.5}}
    \put(0,40){\circle*{3.5}}
    \put(20,0){\circle*{3.5}}   
    \put(20,60){\circle*{3.5}} 
    \put(40,0){\circle*{3.5}}    
    \put(40,60){\circle*{3.5}} 
    \put(60,20){\circle*{3.5}}    
    \put(60,40){\circle*{3.5}} 
            
    \put(20,60){\line(1,0){20}}
    \put(24,-2.5){$\cdots$}   
    \put(0,20){\line(0,1){20}}  
    \put(60,20){\line(0,1){20}}
    \put(0,40){\line(1,1){20}} 
    \put(0,20){\line(1,-1){20}}
    \put(40,0){\line(1,1){20}}   
    \put(40,60){\line(1,-1){20}}  
    
    \put(27,64){$\scriptstyle 4$}
    \put(-9,27){$\scriptstyle 5$}
    \put(64,27){$\scriptstyle 5$}   
    \put(2,5){$\scriptstyle 4$}
    \put(52,5){$\scriptstyle 4$}
                
\end{picture}
\vskip.8cm

\caption{The Vinberg graph of Tumarkin's polyhedron $T\subset \mathbb H^5$ with one pair of disjoint facets
}
\label{fig:Tumarkin}
\end{figure}
\end{example}

\setcounter{rmk}{0}
\begin{rmk}\label{one-dotted}
By a result of Felikson and Tumarkin \cite[Corollary]{FT2}, the Coxeter polyhedra in Examples \ref{ex1}, \ref{n+2} and \ref{n+3} contain {\em all} compact Coxeter polyhedra with exactly one pair of non-intersecting facets. In particular, each compact Coxeter polyhedron $P\subset\mathbb H^n$ with $N\ge n+4$ facets has a Vinberg graph with at least 2 dotted edges.
\end{rmk}

Every compact Coxeter polyhedron $P\subset\mathbb H^n$ gives rise to a hyperbolic Coxeter group acting cocompactly on $\mathbb H^n$, and each cocompact discrete group $G\subset\hbox{Isom}\mathbb H^n$ generated by finitely many hyperplane reflections has a fundamental domain whose closure is a compact Coxeter polyhedron in $\mathbb H^n$. In the sequel, we often use identical notions and descriptions for both, the polyhedron $P$ and the reflection group $G$.

\smallskip
For further details and results about hyperbolic Coxeter 
polyhedra and Coxe\-ter groups, their geometric-combinatorial and arithmetical characterisation as well as general (non-)existence results,
we refer to the foundational work of E. Vinberg \cite{V1, V2}.  An overview about the diverse partial classification 
results can be found in \cite{F-web}.

\subsection{Coxeter groups and growth rates}\label{section2-2}
A hyperbolic Coxeter group $G=(G,S)$ with $S=\{r_1,\ldots,r_N\}$  as 
above is the geometric realisation of an abstract Coxeter system $(W,S)$ of rank $N$
consisting of a group $W$ generated by a subset $S$ of elements $s_1,\ldots,s_N$ satisfying the relations as given by \eqref{refl-relations}. 
In the fundamental work \cite{Coxeter} of Coxeter,  the irreducible finite (or spherical) and affine Coxeter groups are classified.
Abstract Coxeter groups are most conveniently described by their \emph{Coxeter graphs} or by their {\it Coxeter symbols}. 
More precisely, the Coxeter graph 
$\Sigma=\Sigma(W)$ of a Coxeter system $(W,S)$ has nodes $v_1,\dots,v_N$ corresponding to the generators $s_1,\ldots,s_N$
of $W$, and two nodes $v_i$ and $v_j$ are joined by an edge with weight $m_{ij}\ge3$. In particular, there will be no edge if $m_{ij}=2$
and there will be an edge decorated by $\infty$ if the product element $s_is_j$ is of infinite order $m_{ij}=\infty$. 

In this way, the Vinberg 
graph of a hyperbolic Coxeter group is a refined version of its Coxeter graph.  
In this context, observe that the Coxeter graph $\grinf\bullet$ describes the affine group $\widetilde{A}_1$ and - simultaneously - is underlying the Vinberg graph $\bullet\cdots\bullet$ of a compact hyperbolic Coxeter $1$-simplex as given by any geodesic segment.
Furthermore, the reflection group in $\hbox{Isom}\mathbb H^2$ associated to the compact Lambert quadrilateral with Vinberg graph
$\bullet\cdots\gd\kern-2pt\bullet\cdots\bullet$ is given by the Coxeter graph $\grinf\gd\grinf\bullet$ while the Vinberg graph $\grinf\gd\bullet$ (coinciding with its Coxeter graph) describes
a non-compact hyperbolic triangle of area $\frac{\pi}{6}$.

\vskip.1cm
In the case that the rank $N$ of the Coxeter system $(W,S)$ is small, a description by the Coxeter symbol is more convenient. For example, $[p_1,\ldots,p_k]$ with integer labels $p_i\ge3$ is associated to a linear Coxeter graph
with $k+1$ edges marked by the respective weights. 
The Coxeter symbol $[(p,q,r)]$ describes a cyclic Coxeter graph with 3 edges of weights $p$, $q$ and $r$. We assemble
the different symbols into a single one in order to describe the different nature of parts of the
Coxeter graph in question; see 
\cite[Appendix]{JKRT}, for example. 
\begin{example}\label{Esselmann}
The Coxeter symbols of the seven Esselmann polyhedra in $\mathbb H^4$ are characterised by the fact that they contain two disjoint Coxeter symbols associated to compact hyperbolic triangles and called {\it triangular components} that are separated by at least one edge of (finite) weight $m\ge3$. Accordingly, the Esselmann polyhedron $E\subset\mathbb H^4$ as depicted in Figure \ref{fig:n+2} is described by the Coxeter symbol $[(3,4,3),4,(3,4,3)]$. Notice that none of the triangular components $(p,q,r)$, given by integers $p,q,r\ge2$ such that $\,\frac{1}{p}+\frac{1}{q}+\frac{1}{r}<1$, of the Coxeter symbols appearing in Esselmann's list is equal to $(2,3,7)$.
\end{example}

\medskip
For a Coxeter system $(W,S)$ with generating set $S=\{s_1,\ldots,s_N\}$, the (spherical) growth series $f_S(t)$ is defined by
\[
f_S(t)=1+\sum\limits_{k\ge1}a_kt^k\,\,,
\]
where $a_k\in\mathbb Z$ is the number of words $w\in W$ with $S$-length $k$. For references of the subsequent basic properties of $f_S(t)$, see for example
\cite{Hum, KK, KP}.  
The series $f_S(t)$ can be computed by Steinberg's formula
\begin{equation}\label{eq:Steinberg}
\frac{1}{f_S(t^ {-1})}=\sum\limits_{{W_T<W\atop\scriptscriptstyle{{\vert W_T\vert<\infty}}}}\,
\frac{(-1)^{\vert T\vert}}{f_T(t)}\,,
\end{equation}
where $W_T\,,\,T\subset S\,,$ is a finite Coxeter subgroup of $W$, and where $W_{\varnothing}=\{1\}$.
By a result of Solomon, the growth polynomials $f_T(t)$ in \eqref{eq:Steinberg} can be expressed by means of their \emph{exponents} $m_1=1,m_2,\ldots,m_p$
according to the formula
\begin{equation}\label{eq:Solomon}
f_T(t)=\prod\limits_{i=1}^ {p}\,[m_i+1]\,.
\end{equation}
Here we use the standard notation $\,[k]=1+t+\dots+ t^{k-1}$ with $\,[k,l]=[k]\cdot[l]\,$ and so on. By replacing the variable $t$ by $t^{-1}$,
the function $[k]$ 
satisfies the property $\,[k](t)=t^{k-1}[k](t^{-1})$.
\bigskip
\begin{table}
\vbox{\offinterlineskip
\hrule
\halign{&\vrule#&
\strut\quad#\quad\cr
height6pt&\omit&&\omit&&\omit&\cr
&\hfil Group\hfil&&Exponents\hfil&&\hfil Growth polynomial $f_S(x)$\hfil&\cr
height4pt&\omit&&\omit&&\omit&\cr
\noalign{\hrule}
height6pt&\omit&&\omit&&\omit&\cr
&$A_n$\hfil&&${1,2,\ldots,n-1,n}$&&$[2,3,\ldots,n,n+1]$\hfil&\cr
height3pt&\omit&&\omit&&\omit&\cr
&$B_n$&&$1,3,\ldots,2n-3,2n-1$\hfil&&$[2,4,\ldots,2n-2,2n]$\hfil&\cr
height3pt&\omit&&\omit&&\omit&\cr
&$D_n$&&${1,3,\ldots,2n-5,2n-3,n-1}$&&$[2,4,\ldots,2n-2,n]$\hfil&\cr
height3pt&\omit&&\omit&&\omit&\cr
&$G_2^{(m)}$\hfil&&$1,m-1$\hfil&&$[2,m]$\hfil&\cr
height3pt&\omit&&\omit&&\omit&\cr
&$F_4$\hfil&&$1,5,7,11$\hfil&&$[2,6,8,12]$\hfil&\cr
height3pt&\omit&&\omit&&\omit&\cr
&$H_3$\hfil&&$1,5,9$\hfil&&$[2,6,10]$\hfil&\cr
height3pt&\omit&&\omit&&\omit&\cr
&$H_4$\hfil&&$1,11,19,29$\hfil&&$[2,12,20,30]$\hfil&\cr
height6pt&\omit&&\omit&&\omit&\cr}
\hrule}
\bigskip
\caption{Exponents and growth polynomials of irreducible finite Coxeter groups}
\label{tab:table1}
\end{table}

\smallskip
Table \ref{tab:table1}
lists all irreducible finite Coxeter groups together with their growth polynomials up to the exceptional groups $E_6,E_7$ and $E_8$ which are irre\-levant for this work.
Let us add that the growth series of a reducible Coxeter system $(W,S)$ with factor groups $(W_1,S_1)$ and $(W_2,S_2)$ such that
$S=(S_1\times\{1_{W_2}\})\cup(\{1_{W_1}\}\times S_2)$, satisfies the product formula $f_S(t)=f_{S_1}(t)\cdot f_{S_2}(t)$.

By the above, in its disk of convergence, the growth series $f_S(t)$ is a rational function that can be expressed as the quotient of coprime monic polynomials $p(t), q(t)\in\mathbb Z[t]$ of equal degree. The {\em growth rate} $\tau_W=\tau_{(W,S)}$ is defined by 
\begin{equation*}
\tau_W=\limsup\limits_{k\rightarrow\infty}\sqrt[\leftroot{0}\uproot{2}k]{a_k}\,\,,
\end{equation*}
and it coincides with the inverse of the radius of convergence $R$ of $f_S(t)$. Since $\tau_W$ equals the biggest real root of the denominator polynomial $q(t)$, it is a real algebraic integer.

Consider a cocompact hyperbolic Coxeter group $G=(G,S)$. Then, 
the rational function $f_S(t)$ is reciprocal (resp. anti-reciprocal) for $n$ even (resp. $n$ odd)
(see \cite{KP}, for example). 
In particular, for $n=2$ and $4$, one has
$f_S(t^ {-1})=f_S(t)$ for all $t\not=0$. 
Furthermore, a
result of Milnor \cite{Milnor} implies that the growth rate $\tau_G$ is strictly bigger than $1$ so that $G$ is of exponential growth.
More specifically, for $n=2$ and $3$, $\tau_G$ is either a quadratic unit or a {\em Salem number}, that is, $\tau_G$ is a real algebraic integer $\alpha>1$ whose inverse
is a conjugate of $\alpha$, and all other conjugates lie on the unit circle (see \cite{Komori}, for example).  However,  by a result of 
Cannon \cite{Cannon, Cannon-Wag} 
(see also \cite[Theorem 4.1]{KP}), the growth rates of the five
Lann\'er groups acting on $\mathbb H^4$ and shown in Figure \ref{fig:cox-simplex} are not Salem numbers anymore; they 
are so-called {\em Perron numbers}, that is, real algebraic integers 
$>1$ all of whose other conjugates are of strictly smaller absolute value. 

\begin{example}\label{mingrowth-dim2}
The smallest known Salem number $\alpha_L\approx1.176281$ with minimal polynomial 
$L(t)=t^{10}+t^9-t^7-t^6-t^5-t^4-t^3+t+1$ equals the growth rate $\tau_{[7,3]}$ of the cocompact Coxeter triangle 
group $G=[7,3]$
with Coxeter graph $\gr{7}\gd\bullet$
which in turn is the smallest growth rate among {\em all} cocompact planar hyperbolic Coxeter groups; see \cite{Hironaka,KK}. 

The second smallest growth rate among them is realised by the Coxeter triangle group $[8,3]$ with Coxeter graph $\gr{8}\gd\bullet$ and appears as the seventh 
smallest known Salem number
$\approx1.23039$ given by the minimal polynomial $t^{10}-t^7-t^5-t^3+1$; see \cite{KL}.

As a consequence, the growth rate of the cocompact Lambert quadrilateral group $Q$ with Vinberg graph
$\,\bullet\cdots\gd\kern-2pt\bullet\cdots\bullet\,$ is strictly bigger than $\tau_{[8,3]}$. 
More precisely,
the growth rate of $Q$ is the Salem number $\tau_{Q}\approx 1.72208$ with minimal polynomial $t^4-t^3-t^2-t-1$. Notice also that the Coxeter graph of $Q$ equals $\,\grinf\gd\grinf\bullet\,$ (see Section \ref{section3-2}).

By applying similar techniques, it was shown in \cite{K0} (see also Floyd's work \cite{Fl}) that the Coxeter triangle group with Vinberg graph $\grinf\gd\bullet$ has
smallest growth rate among all non-cocompact hyperbolic Coxeter groups of finite coarea in $\hbox{Isom}\mathbb H^2$, and as such it is unique. The growth rate $\tau_{[\infty,3]}\approx 1.32471$ has minimal polynomial $t^3-t-1$ and
equals the smallest Pisot number $\alpha_S$ as shown by C. Smyth (see \cite{Smyth} and \cite[Section 3.2]{K0}, for example). Recall that a {\it Pisot number} is an algebraic integer $\alpha>1$ all of whose other conjugates are of absolute value less than $1$.

For later purpose,  let us emphasize the above comparison result as follows.
\begin{equation}\label{eq:tau-[83]}
\tau_{[8,3]}<\tau_{[\infty,3]}\,\,.
\end{equation}
 \end{example}
 
 \begin{example}\label{mingrowth-dim3}
Among the cocompact Coxeter tetrahedral groups, the smallest growth rate is about $1.35098$ with minimal polynomial 
$t^{10}-t^9-t^6+t^5-t^4-t+1$; it is achieved in a unique way by the
group $G=[3,5,3]$ with Coxeter graph $\gd\gr{5}\gd\bullet$; see \cite{KK}.
 \end{example}
 
\begin{example}\label{growth-5333}
Consider the (arithmetic) Lann\'er group $L=[5,3,3,3]$ with Coxeter graph $\,\gr{5}\gd\gd\gd\bullet$ mentionned in Example \ref{ex1}. 
By means of Steinberg's formula \eqref{eq:Steinberg} and Table \ref{tab:table1}, the growth function $f_L(t)=f_S(t)$ can be expressed according to
\begin{align*}\label{eq:functionL}
 &\frac{1}{f_L(t^{-1})}=\frac{1}{f_L(t)}=1-\frac{5}{[2]}+\frac{6}{[2,2]}+\frac{3}{[2,3]}+\frac{1}{[2,5]}-&\nonumber\\
 &-\left\{\frac{1}{[2,2,2]}+\frac{4}{[2,2,3]}+\frac{2}{[2,2,5]}+\frac{2}{[2,3,4]}+\frac{1}{[2,6,10]}\right\}+\\
 &+\frac{1}{[2,2,3,4]}+\frac{1}{[2,2,3,5]}+\frac{1}{[2,2,6,10]}+\frac{1}{[2,3,4,5]}+\frac{1}{[2,12,20,30]}\,\,.&\nonumber
\end{align*}

\smallskip
It follows that $f_L(t)$ equals
\[
\frac{[2, 12, 20, 30]}{\begin{array}{l}
1 - t - t^7 + t^8 - t^9 + t^{10} - t^{11} + t^{14} - t^{15} + t^{16} - 2t^{17} + 2t^{18} - t^{19} + t^{20} \\
-\ t^{21} + t^{22} - t^{23} + 2t^{24} - 2t^{25} + 2t^{26} - 2t^{27} + 2t^{28} - t^{29} + t^{30} - t^{31} + 2t^{32} \\
-\ 2t^{33} + 2t^{34} - 2t^{35} + 2t^{36} - t^{37} + t^{38} - t^{39} + t^{40} - t^{41} + 2t^{42} - 2t^{43} + t^{44} \\
-\ t^{45} + t^{46} - t^{49} + t^{50} - t^{51} + t^{52} - t^{53} - t^{59} + t^{60}
\end{array}} \,\,.
\]

\end{example}

The denominator polynomial $q(t)$ of $f_L(t)$ is palindromic and of degree 60. By means of the software PARI/GP \cite{PARI2}, one checks that $q(t)$
is irreducible and has -- beside non-real roots some of them being of absolute value one --
exactly two inversive pairs $\alpha^{\pm 1}, \beta^{\pm 1}$
of real roots such that $\alpha>\beta>1$. Indeed, by the results in \cite{Cannon,Cannon-Wag}, $\alpha$ is not a 
Salem number anymore. As a consequence, the growth rate $\tau_L=\alpha\approx1.19988$ of the Lann\'er group $L=[5,3,3,3]$ is {\it not} a 
Salem number. However, $\tau_{[5,3,3,3]}$ is a Perron number. All these properties can be checked by the software CoxIter developed by R. Guglielmetti \cite{Gug1, Gug2}.

\begin{example}\label{growth-53333}
The Coxeter prism $M\subset\mathbb H^5$ found by Makarov is given by the Vinberg graph 
$\,\gr{5}\gd\gd\gd\gd\kern-1pt\bullet\cdots\bullet\kern-14pt\raise
6pt\hbox{$\scriptstyle l$}\kern12pt$ where the hyperbolic distance $l$ between the (unique) pair of non-intersecting facets of $M$ satisfies
\[
\cosh l=\frac{1}{2}\,\sqrt{\frac{7+\sqrt{5}}{2}}\approx1.07448\,\,.
\]
In fact, the computation of $l$ is easy since the determinant of the Gram matrix of $M$ vanishes.
As in Example \ref{growth-5333}, one can exploit Steinberg's formula \eqref{eq:Steinberg} and Table \ref{tab:table1} in order to establish the growth function $f_M(t)$. The denominator polynomial of $f_M(t)$ splits into the factor $t-1$ and a certain irreducible palindromic polynomial $q(t)$. As above, the software CoxIter allows us to identify the growth rate of the reflection group $[5,3,3,3,3]$ associated to $M$, as given by the largest zero of $q(t)$, with the Perron number $\tau_M\approx1.64759$. Notice that the factor $t-1$ is responsible for the vanishing of the Euler characteristic of $M$ (see \cite[(2.7)]{KK}, for example).
\end{example}
\begin{example}\label{growth-K}
For the Kaplinskaja prism $K\subset\mathbb H^5$ depicted in Figure
\ref{fig:n+2}, the denominator polynomial of the growth function $f_K(t)$ splits into the factor $t-1$ and an irreducible palindromic polynomial $q(t)$ of degree $32$. By means of CoxIter, one deduces that
the growth rate is a Perron number of value $\tau_K\approx 2.08379$.
\end{example}
\medskip
In a similar way, one computes the individual growth series and related invariants and properties of any cocompact (or cofinite) hyperbolic Coxeter group with given Vinberg graph. 

\medskip
Growth rates satisfy an important monotonicity property on the partially ordered set of Coxeter systems as follows. 
For two Coxeter systems $(W,S)$ and $(W',S')$, one defines $(W,S)\le(W',S')$ if there is an injective map $\iota:S\rightarrow S'$ such that $m_{st}\le m'_{\iota(s)\iota(t)}$ for all $s,t\in S$. If $\iota$ extends to an isomorphism between $W$ and $W'$, one writes $\,(W,S)\simeq(W',S')$, and $\,(W,S)<(W',S')$ otherwise. 
This partial order satisfies the descending chain condition since $m_{st}\in\{2,3,\ldots,\infty\}$ where $s\not=t$. In particular,
any strictly decreasing sequence of Coxeter systems is finite (see \cite{McM}).
In this work, the following result of Terragni \cite[Section 3]{Terragni}
will play an essential role.
\setcounter{thm}{1}
\begin{thm}\label{Terragni}
If $\,(W,S)\le(W',S')$, then  $\tau_{(W,S)}\le\tau_{(W',S')}$.
\end{thm}

\begin{example}\label{ex:tau-Esselmann}
Consider the seven Esselmann groups $E_i\subset\hbox{Isom}\mathbb H^4\,,\,1\le i\le 7\,,$ whose Coxeter symbols consist of two triangular components separated by at least one edge of weight $m\ge3$; see Example \ref{Esselmann}. Each of their triangular components describes a cocompact Coxeter group in $\hbox{Isom}\mathbb H^2$ of the type $(2,3,8)$, $(2,3,10)$, $(2,4,5)$, $(2,5,5)$, $(3,3,4)$ or $(3,3,5)$. By means of Theorem \ref{Terragni}, we conclude that 
\begin{equation}\label{eq:tau_E}
\tau_{[8,3]}\le\tau_{E_i}\,\,,\,\,1\le i\le 7\,\,.
\end{equation}
\end{example}

\begin{notice*} In the sequel, we will work with the Coxeter graph instead of the Vinberg graph associated to a hyperbolic Coxeter group $(W,S)$. Hence, we replace each dotted edge between two nodes $\nu_s$ and $\nu_{s'}$ by an edge with weight $\infty$, just indicating that 
the product element $ss'\in W$ is of infinite order.  
\end{notice*}
\section{Growth minimality in dimensions four and five}\label{section3}
In this section, we prove the following two results as announced in the Introduction  (see Section \ref{Intro}).

\begingroup
\def\thetheorem{\ref{thm:A}}
\begin{theorem}
Among all Coxeter groups acting cocompactly on $\mathbb H^4$, the 
Coxeter simplex group $[5,3,3,3]$ has minimal growth rate, and as such it is unique.
\end{theorem}
\addtocounter{theorem}{-1}
\endgroup

\begingroup
\def\thetheorem{\ref{thm:B}}
\begin{theorem}
Among all Coxeter groups acting cocompactly on $\mathbb H^5$, the 
Coxeter prism group based on $[5,3,3,3,3]$ has minimal growth rate, and as such it is unique.
\end{theorem}
\addtocounter{theorem}{-1}
\endgroup

\subsection{The proof of Theorem A}\label{section3-1}
Consider a group $G\subset\hbox{Isom}\mathbb H^4$ generated by the set $S$ of reflections $r_1,\ldots,r_N\,$ in the $N$ facet hyperplanes bounding  a compact Coxeter polyhedron $P\subset\mathbb H^4$.
The group $G=(G,S)$ is a cocompact hyperbolic Coxeter group of rank $N\ge5$.  Assume that the group $G$ is not isomorphic to the Coxeter simplex group $L=[5,3,3,3]$. We have to show that $\tau_G>\tau_{[5,3,3,3]}\approx1.19988$.

\medskip
\parindent=0pt
In view of Theorem \ref{1-dotted}, we distinguish between the two cases whether all facets of $P$ are mutually intersecting or not.
In the case that all facets of $P$ are mutually intersecting, then $P$ is either a Lann\'er simplex and $G$ is of rank $5$, or $P$ is one of the seven Esselmann polyhedra with related Coxeter groups $E_i\,,\,1\le i \le 7\,,$ of rank $6$.

\medskip
\parindent=0pt
(1a) The Coxeter graphs of the five Lann\'er simplices $L=L_1,\ldots,L_5$ in $\mathbb H^4$  are given in Figure \ref{fig:cox-simplex}. 
The associated growth rates have been computed by means of Steinberg's formula and are well known (see also \cite{Cannon}, \cite{Perren}, \cite{Terragni2}). The software CoxIter yields the values
\begin{align*}
\tau_{[5,3,3,4]}\approx 1.38868\quad,&\quad\tau_{[5,3,3,5]}\approx1.51662\quad,\\
\tau_{[5,3,3^{1,1}]}\approx 1.44970\quad,&\quad\tau_{[(3^4,4)]}\approx 1.62282\quad,
\end{align*}
implying that the growth rate of $L=[5,3,3,3]$ is strictly smaller than those of $L_2,\ldots,L_5$.

\smallskip
\parindent=0pt
(1b) Let us investigate the growth rates of the Esselmann groups $E_1,\ldots,E_7$. By 
Example \ref{ex:tau-Esselmann}, \eqref{eq:tau_E}, we have that
\begin{equation*}
\tau_{[8,3]}\le\tau_{E_i}\,\,,\,\,1\le i\le 7\,\,.
\end{equation*}
It follows from Example \ref{mingrowth-dim2} and Example \ref{growth-5333}  that
\[
1.19988\approx\tau_{[5,3,3,3]}<1.2<\tau_{[8,3]}\approx1.23039\,\,,
\]
which shows that the growth rate of $L=[5,3,3,3]$ is strictly smaller than those of the Esselmann groups $E_1,\ldots,E_7$.

\medskip
\parindent=0pt
(2) Suppose that $P$ has at least one pair of non-intersecting facets. Therefore, the Coxeter graph $\Sigma$ of $P$ contains at least one edge with weight $\infty$. Since $P$ has at least $N\ge6$ facets,
the graph $\Sigma$ -- being connected -- contains a proper connected subgraph $\sigma$ of order $3$ with weights $p,q\in\{2,3,\ldots,\infty\}$ of the form as depicted in Figure \ref{fig:sigma}.
\begin{figure}[H]
\centering
\begin{picture}(40,30)
    \put(0,0){\circle*{3.5}}
    \put(30,0){\circle*{3.5}}
    \put(15,23){\circle*{3.5}}   
    
    \put(0,0){\line(1,0){30}} 
    \put(0,0){\line(2,3){15}}  
    \put(30,0){\line(-2,3){15}}  
             
    \put(10,-9){$\infty$}
    \put(-5,10){$p$}
    \put(28,10){$q$}
      
\end{picture}
\caption{A subgraph $\sigma$ of $\Sigma$}
\label{fig:sigma}
\end{figure}
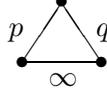
By construction, the subgraph $\sigma$ gives rise to a standard Coxeter subgroup $(W,T)$ of rank $3$ of $(G,S)$ that 
satisfies $(W,T)\le(G,S)$. By Theorem \ref{Terragni}, Example \ref{mingrowth-dim2}, \eqref{eq:tau-[83]}, and Example \ref{growth-5333}, we deduce in a similar way as above that
\[
\tau_{[5,3,3,3]}<\tau_{[8,3]}<\tau_{[\infty,3]}\le \tau_{\sigma}\le\tau_{\Sigma}\,\,,
\]

which finishes the proof of Theorem \ref{thm:A}.

\hfill{\qedsymbol}

\bigskip
\subsection{The proof of Theorem B}\label{section3-2}
Let $G\subset\hbox{Isom}\mathbb H^5$ be a discrete group
generated by the set $S$ of reflections $r_1,\ldots,r_N\,$ in the $N$ facet hyperplanes of a compact Coxeter polyhedron $P\subset\mathbb H^5$.
The group $G=(G,S)$ is a cocompact hyperbolic Coxeter group of rank $N\ge6$.  Assume that $G$ is not isomorphic to Makarov's rank $7$ prism group based on $[5,3,3,3,3]$. The associated
Coxeter prism $M$ is described and the growth rate $\tau_M$
is given in Example \ref{growth-53333}.
We have to show that $\tau_G>\tau_M\approx1.64759$.

\medskip
\parindent=0pt
Inspired by the proof of Theorem \ref{thm:A},
we look for appropriate Coxeter groups of smaller rank such that their growth data can be exploited to derive suitable lower bounds
in view of Theorem \ref{Terragni}.
To this end, consider the following abstract Coxeter groups $W_1, W_2$ and $W_3\,$
with generating subsets $S_1,S_2$ and $S_3$ 
of rank $4$ as defined by the Coxeter graphs in Figure $\ref{fig:Cox-graphs1}$.
\begin{figure}[H]
\centering
\begin{picture}(280,35)
    \put(0,20){\circle*{3.5}}
    \put(20,20){\circle*{3.5}}
    \put(40,20){\circle*{3.5}}   
    \put(60,20){\circle*{3.5}} 

    \put(110,20){\circle*{3.5}}
    \put(130,20){\circle*{3.5}}
    \put(150,20){\circle*{3.5}}   
    \put(170,20){\circle*{3.5}} 
       
    \put(220,20){\circle*{3.5}}
    \put(240,20){\circle*{3.5}}
    \put(260,20){\circle*{3.5}}   
    \put(240,0){\circle*{3.5}}
    
    \put(0,20){\line(1,0){60}}
    \put(110,20){\line(1,0){60}}
    \put(220,20){\line(1,0){40}}
    \put(240,0){\line(0,1){20}}

    \put(6,24){$\scriptstyle \infty$}
    \put(46,24){$\scriptstyle \infty$}
    \put(116,24){$\scriptstyle \infty$}
    \put(136,24){$\scriptstyle \infty$} 
    \put(226,24){$\scriptstyle \infty$}  
    \put(246,24){$\scriptstyle \infty$}   
    
\end{picture}
\caption{The three abstract Coxeter groups $W_1,W_2$ and $W_3$}
\label{fig:Cox-graphs1}
\end{figure}
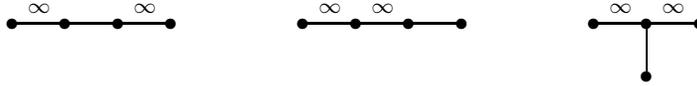
The Coxeter systems $(W_i,S_i)$ can be represented by hyperbolic Coxeter groups $G_i$ for each $1\le i\le 3$, and they will play an important role when comparing growth rates.
In fact, the Coxeter graph of $W_1$ coincides with the Coxeter graph of the 
cocompact Lambert quadrilateral group $Q\subset\hbox{Isom}\mathbb H^2$ with growth rate $\tau_Q\approx1.72208$ (see Example \ref{mingrowth-dim2}). Since, for the Makarov prism $M$, we have $\tau_M\approx1.64759$,
we deduce the following important fact.
\begin{equation}\label{comp-value5}
\tau_M<\tau_Q=\tau_{G_1}\,\,.
\end{equation}
Each of the remaining Coxeter groups $W_2$ and $W_3$ can be represented as a discrete subgroup of $O^+(3,1)$ generated by reflections 
in the facets of a Coxeter tetrahedron of infinite volume. Indeed, one easily checks that the associated Tits form is of signature $(3,1)$ and that some of the simplex vertices are not hyperbolic but ultra-ideal points (of positive Lorentzian norm).
More importantly, the following result holds.
\setcounter{lem}{0}
\begin{lem}\label{compare1}$ $\newline
\hbox{\rm (1)\quad$\tau_{G_1}<\tau_{G_2}$\,;}

\hbox{\rm (2)\quad$\tau_{G_1}<\tau_{G_3}$\,.}
\end{lem}
\begin{proof}
By means of Steinberg's formula \eqref{eq:Steinberg}, we identify for each $G_i$ the finite Coxeter subgroups with their growth polynomials according to Table \ref{tab:table1} in order to deduce the following expressions for their growth functions $f_i(t)\,,\,1\le i\le 3$.
\begin{align}
\frac{1}{f_1(t^{-1})}&= h(t)\tag{a}\label{eq:g1}\\
\frac{1}{f_2(t^{-1})}&= h(t)-\frac{1}{[2,2,3]}\tag{b}\label{eq:g2}\\
\frac{1}{f_3(t^{-1})}&=h(t)-\frac{1}{[2,2,2]}\tag{c}\label{eq:g3}
\end{align}
Here, the help function $h(t)\,,\,t\not=0\,,$ is given by 
\begin{equation}\label{eq:help}
h(t)=1-\frac {4}{[2]}+\frac {3}{[2,2]}+\frac {1}{[2,3]}\,.
\end{equation}
By taking the differences between \eqref{eq:g1} and \eqref{eq:g2}, \eqref{eq:g3}, respectively, one obtains, for all $t>0$, 
\[
\frac{1}{f_1(t^{-1})}-\frac{1}{f_2(t^{-1})}=\frac{1}{[2,2,3]}>0\quad,\quad\frac{1}{f_1(t^{-1})}-\frac{1}{f_3(t^{-1})}=\frac{1}{[2,2,2]}>0\,\,.
\]
For $x=t^{-1}\in(0,1)$, we deduce that the smallest zero of $1/f_1(x)$ as given by the radius of convergence of the growth series $f_1(x)$ of $G_1$ 
is strictly bigger than the one of $1/f_2(x)$
and of $1/f_3(x)$. Hence, we get $\tau_{G_1}<\tau_{G_2}$ and $\tau_{G_1}<\tau_{G_3}$.
\end{proof}
For later use, we also compare the growth rate of $W_1=Q$ with the one of the Coxeter group
$W_4$ with generating subset $S_4$ of rank $4$ given by the Coxeter graph according to Figure \ref{fig:Cox-graphs2}.  Again, the group $W_4$ can be interpreted as a discrete subgroup
$G_4\subset O^+(3,1)$ generated by the reflections 
in the facets of a Coxeter tetrahedron of infinite volume. 
\begin{figure}[H]
\centering
\begin{picture}(60,35)
    \put(0,20){\circle*{3.5}}
    \put(20,20){\circle*{3.5}}
    \put(40,20){\circle*{3.5}}   
    \put(60,20){\circle*{3.5}} 
    
    \put(0,20){\line(1,0){60}}
    \put(6,24){$\scriptstyle 4$}
    \put(26,24){$\scriptstyle \infty$} 
    \put(46,24){$\scriptstyle 4$} 
\end{picture}
\vskip-.5cm
\caption{The abstract Coxeter group $W_4$}
\label{fig:Cox-graphs2}
\end{figure}
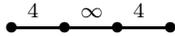
\begin{lem}\label{compare2}
$\quad \tau_{G_1}<\tau_{G_4}\,.$
\end{lem}
\begin{proof}
We proceed as in the proof of Lemma \ref{compare1} and establish the growth function $f_4(t)$ by means of Steinberg's formula. We obtain the following expression.
\begin{align*}
\frac{1}{f_4(t^{-1})}= 1-\frac {4}{[2]}+\frac {3}{[2,2]}+\frac {2}{[2,4]}-\frac{2}{[2,2,4]}\,\,.\tag{d}\label{eq:g4}\\
\end{align*}

By means of \eqref{eq:g1}, \eqref{eq:g4} and \eqref{eq:help}, we obtain the difference function
\begin{equation*}
\frac{1}{f_1(t^{-1})}-\frac{1}{f_5(t^{-1})}=\frac{1}{[2,3]}-\frac{2}{[2,4]}+\frac{2}{[2,2,4]}=\frac{t^4+1}{[2,3]\,(t^2+1)}
>0\quad,\quad\forall t>0\,\,,
\end{equation*}
and conclude as at the end of the previous proof.
\end{proof}

Let us return and consider a compact Coxeter polyhedron $P\subset\mathbb H^5$ with $N$ facets and associated hyperbolic Coxeter group $G$. By Example \ref{ex1}, we know that there are no compact Coxeter simplices anymore so that $N\ge7$. Furthermore, by Theorem \ref{1-dotted}, $P$ has at least one pair of non-intersecting facets. In the sequel, we discuss the cases $N=7$, $N=8$ and $N\ge9$.

\medskip
For $N=7$, we are left with the three Kaplinskaja prisms (and their glueings) as given by the Makarov prism $M=:M_3$ based on $[5,3,3,3,3]$, its closely related Coxeter prism $M_4$ based on $[5,3,3,3,4]$ as well as the Coxeter prism $K$ with Vinberg graph depicted in Figure \ref{fig:n+2} and treated in Example \ref{growth-K}. By means of the software CoxIter (or some lengthy computation), one obtains the growth rate inequalities
\[
1.64759\approx\tau_M<\tau_{M_4}<1.84712<\tau_K\approx2.08379\,\,,
\]
 which confirm the assertion of Theorem \ref{thm:B} in this case.
 
\medskip
For $N=8$, we dispose of Tumarkin's classification list comprising all compact Coxeter polyhedra with $n+3$ facets. For $n=5$, these polyhedra have Vinberg graphs with exactly three (consecutive) dotted edges except for the polyhedron $T\subset\mathbb H^5$ depicted in Figure \ref{fig:Tumarkin}. 

The Coxeter graph associated to $T$ contains the proper subgraph $\gr{4}\grinf\gr{4}\bullet$ which is associated to the Coxeter group $W_4$ studied above; see Figure \ref{fig:Cox-graphs2}.
By means of Theorem \ref{Terragni}, Lemma \ref{compare2} and \eqref{comp-value5}, we deduce that
\[
\tau_M<\tau_{G_4}\le\tau_T\,\,.
\]
For the Coxeter graph of a polyhedron $P$ with $8$ facets in $\mathbb H^5$ that is not isometric to $T$, we consider its proper order 4 subgraph $\,\grinf\grinf\grinf\bullet\,$. In a similar way, by
Theorem \ref{Terragni}, Lemma \ref{compare1} and \eqref{comp-value5}, we obtain
\[
\tau_M<\tau_Q\le\tau_P\,\,.
\]

\medskip
Let $N\ge9$. By Remark \ref{one-dotted}, the Vinberg graph of the polyhedron $P\subset\mathbb H^5$ with $N$ facets has at least two dotted edges. However, two dotted edges are separated by an edge in view of the signature condition of the Gram matrix $\Gr(P)$; see Section \ref{section2-1}.

\smallskip
Consider the Coxeter graph $\Sigma$ of order $N$ of the hyperbolic Coxeter group $G$ associated to $P$. By the above, there is a proper connected subgraph $\sigma$ of order $4$ in $\Sigma$, depicted in Figure \ref{fig:sigma-5}, with weights $\,p,q,r,s,t\in\{2,3,\ldots,\infty\}$ where at least one of them is equal to $\infty$.
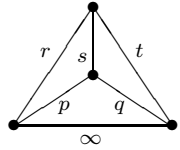
\begin{figure}[H]
\centering
\begin{picture}(60,55)
    \put(0,0){\circle*{3.5}}
    \put(60,0){\circle*{3.5}}
    \put(30,45){\circle*{3.5}} 
    \put(30,19){\circle*{3.5}}  
    
    \put(0,0){\line(1,0){60}}
     \put(0,0){\line(2,3){30}}
        \put(1,0){\line(3,2){29}}
          \put(59,0){\line(-3,2){29}}
     
       \put(60,0){\line(-2,3){30}}
         \put(30,18){\line(0,1){28}}
          
                 \put(0,0){\line(1,0){60}}
                     \put(0,0){\line(1,0){60}}   
    \put(10,26){$\scriptstyle r$}
    \put(24,24){$\scriptstyle s$}  
    \put(46,26){$\scriptstyle t$} 
        \put(17,6){$\scriptstyle p$}  
    \put(38,6){$\scriptstyle q$} 
    \put(25,-7){$\scriptstyle \infty$} 
   
\end{picture}
\caption{The subgraph $\sigma=\sigma(p,q,r,s,t)$}
\label{fig:sigma-5}
\end{figure}
In view of Figure \ref{fig:Cox-graphs1}, describing the three Coxeter groups $G_1,G_2$ and $G_3$, and by means of Theorem \ref{Terragni}, the growth rate of $\Sigma$, and hence of $P$, can be estimated from below according to 
\[
\tau_{G_i}\le\tau_{\sigma}\le\tau_{\Sigma}\quad\hbox{for at least one } i\in\{1,2,3\}\,\,.
\]
By Lemma \ref{compare1} and \eqref{comp-value5}, we finally obtain that
\begin{equation*}\label{eq:conclusion}
\tau_M<\tau_{G_1}\le\tau_P\,\,,
\end{equation*}

as desired. This finishes the proof of Theorem \ref{thm:B}.

\hfill\qedsymbol



\end{document}